\documentclass{article}

\title{Distance formulas in Bruhat-Tits building of $\mathrm{SL}_d(\mathbb{Q}_p)$}

\author{Dominik Lachman\thanks{dominik.lachman01@upol.cz}\\
Department of Algebra and Geometry, Palack\'{y} University Olomouc, \\17. listopadu 12, 779 00 Olomouc, Czech Republic \\
Department of Algebra, Charles University, \\
So\-ko\-lov\-sk\' a~83, 186 00 Praha~8, Czech Republic}

\def\@testdef #1#2#3{%
  \def\reserved@a{#3}\expandafter \ifx \csname #1@#2\endcsname
 \reserved@a  \else
\typeout{^^Jlabel #2 changed:^^J%
\meaning\reserved@a^^J%
\expandafter\meaning\csname #1@#2\endcsname^^J}%
\@tempswatrue \fi}

\usepackage[utf8]{inputenc}
\usepackage[IL2]{fontenc}
\usepackage[english]{babel}
\usepackage{amssymb}
\usepackage{amsthm}
\usepackage{mathrsfs}
\usepackage{mathtools}
\usepackage{enumerate}
\usepackage{amsmath}

\usepackage[square,sort,comma,numbers]{natbib}

\usepackage[IL2]{fontenc}
\usepackage[english]{babel}
\usepackage{mathrsfs}
\usepackage{mathtools} 	
\usepackage{graphicx}
\usepackage{hyperref}

\newcommand{\SL}{\mbox{\rm SL}}
\newcommand{\GL}{\mbox{\rm GL}}

\begin{document}
\theoremstyle{plain}
\newtheorem{thm}{Theorem}
\newtheorem{lemma}[thm]{Lemma}
\newtheorem{cor}[thm]{Corollary}
\newtheorem{prop}[thm]{Proposition}

\theoremstyle{plain}
\newtheorem{defn}[thm]{Definition}

\theoremstyle{remark}
\newtheorem{obs}[thm]{Observation}
\newtheorem{rem}[thm]{Remark}
\newtheorem{example}[thm]{Example}
\newtheorem{exmp}[thm]{Example}

\maketitle

\begin{abstract}We study the distance on the Bruhat-Tits building
of the group  $\mathrm{SL}_d(\mathbb{Q}_p)$  (and its other combinatorial properties).  Coding its vertices
by certain matrix representatives, we introduce a way how to build formulas
with combinatorial
meanings. 
In Theorem~\ref{distance}, we give an explicit formula for the graph distance $\delta(\alpha,\beta)$ of two vertices $\alpha$ and $\beta$  (without having to specify
their common apartment).Our main result, Theorem 2, then extends the distance formula
to a formula for the smallest total distance of a vertex from a given
finite set of vertices. In the appendix we consider the case of $\SL_2(\mathbb{Q}_p)$ and give a formula for the number of edges shared by two given apartments.
\end{abstract}

\maketitle

\section{Introduction}
We call Bruhat-Tits building any simplicial complexes $\mathscr{B}$ satisfying axioms:
\begin{enumerate}[(i)]
\item\label{ax1} $\mathscr{B}$ is a union of subcomplexes called \textit{apartments}, which are all copies of one affine coxeter complex (certain triangulation of Euclidean space). 
\item\label{ax2} Each pair of simplices have a common apartment.
\item\label{ax3} Whenever two apartments $\mathscr{A}_1$, $\mathscr{A}_2$ intersect in two simplices, there is a simplicial isomorphism sending $\mathscr{A}_1$ to $\mathscr{A}_2$, which fixes the two simplices point-wise.
\end{enumerate}
These buildings were designed to study algebraic reductive groups over local non-Archimedean fields. F. Bruhat and J. Tits in~\cite{BT72} and~\cite{BT84} associated to each such group $G$ a ``building" $\mathscr{B}(G)$ (yet their original concept was a bit different from the axioms ~(\ref{ax1}--\ref{ax3})) on which $G$ acts by simplicial automorphisms (so called strong transitivity condition holds) and hence they provided geometric representations to a vast class of groups. We can view $\mathscr{B}(G)$ as a geometric parametrisation of certain important subgroups of  $G$: there are one-to-one correspondences between the simplices of $\mathscr{B}(G)$ and the parahoric subgroups of $G$ and between the apartments and the maximal split tori of $G$.

Historically, the Bruhat-Tits theory is based on the work of Iwahori and Matsumoto~\cite{IM} who investigated Lie groups over local non-Archimedean fields and on the concept of spherical buildings already introduced by Tits in 1950s - the original motivation was to study exceptional groups over a general field.

The basic application of the theory is simplifying some proofs - buildings provide a way to visualize steps in proofs in geometric figures (e.g., Bruhat decomposition). The Bruhat-Tits theory has classical applications in the field of $p$-adic representation theory \cite{MoB}, $S$-aritmetic groups, non-Archimedean generalizations of super-rigidity theorem \cite{GS} and others. See~\cite{Li} for overview of applications of (general) buildings after half of a century of the existence of the concept. 

In the case of (classical) semisimple linear groups $G=SL_{n+1},SO_{2n+1},Sp_{2n}$, etc., (over some non-Archimedean local field) the buildings $\mathscr{B}(G)$'s are divided into families ($\tilde{A_n}$,$\tilde{B_n}$,$\tilde{C_n}$, etc.) according to the structure of their apartments. These families correspond to the root systems of $G$'s. The concrete choice of the base field affects only branching -- the way how the apartments are glued into the building. It is visible on the two basic cases: $\mathscr{B}(\mathrm{SL}_2(\mathbb{Q}_p))$ is a regular $(p+1)$-ary tree where the apartments are infinite lines and $\mathscr{B}(\mathrm{SL}_3(\mathbb{Q}_p))$ is a $2$-dimensional simplicial complex (which is actually hard to visualize), where each vertex has $2p^2+2p+2$ neighbours and the apartments are (natural) tilings of the plane with equilateral triangles. 

In the article we study some combinatorial properties of $\mathscr{B}(\mathrm{SL}_d(\mathbb{Q}_p))$ (or rather of its $1$-dimensional skeleton) where $d\geq 2$ and $\mathbb{Q}_p$ is the field of $p$-adic rationales. We deal only with the case of $p$-adic fields $\mathbb{Q}_p$ in order to slightly simplify our notations and proofs, although it is likely that the results generalize also to the case of arbitrary local non-Archimedean field.

We use the standard lattice construction of $\mathscr{B}(\mathrm{SL}_d(\mathbb{Q}_p))$ (described in Section~\ref{sec3}). We treat Bruhat-Tits buildings as abstract simplicial complexes which fits our aim which is to study combinatorial properties. 
In Section~\ref{sec4} we establish matrices representatives of vertices in a trivial way: a matrix $A\in\mathrm{GL}_d(\mathbb{Q}_p)$ represents a vertex $A_v$ given by the lattice $\mathrm{Row}_{\mathbb{Z}_p}(A)$. The group $\mathrm{SL}_d$ then acts simply by $\mathrm{SL}_d\ni g:A_v\mapsto (A\cdot g^{-1})_v$. 

In Section~\ref{sec5} we consider following invariants: given two matrices $A,B\in\mathrm{GL}_d(\mathbb{Q}_p)$ define an integer $m_i^{d-i}\equiv m_{A:i}^{B:d-i}$ as the minimum of the
valuations of the determinants of the collection of matrices obtained from $A$ and $B$,
by taking $i$ rows from $A$ and $d-i$ rows from $B$. We observe that the $m_i^{d-i}$'s are invariant under left multiplication $A\rightsquigarrow XA$ by any $X\in\mathrm{GL}_d(\mathbb{Z}_p)$ (the same holds for $B$) and any linear combination of $m_i^{d-i}$'s satisfying certain conditions on coefficients (Proposition~\ref{invarianceCon}) gives us well-defined formula on pairs of vertices.

Having the $m_i^{d-i}$'s in hand we prove (in Section~\ref{sec6}):
\begin{thm}\label{distance}
Graph distance (in $1$-skeleton $\mathscr{B}^1(\SL_d(\mathbb{Q}_p)$) of two vertices $\alpha:=A_v$, $\beta:=B_v$ equals
\begin{equation}\label{eqDistance}
\delta(\alpha,\beta)=m^{d}_{0}+m^{0}_{d}-m^{d-1}_{1}-m^{1}_{d-1}.
\end{equation}
\end{thm}
Further examples of formulas with combinatorial meaning are mentioned  below the proof of Theorem~\ref{distance}. 

Given a matrix $A\in\mathrm{GL}_d(\mathbb{Q}_p)$, some essential information about relative position of the vertex $A_v$ with respect to $I_v$ (e.g. distance) is readable from Smith normal form $D$ of $A$ (which is a diagonal matrix defined via relation $A=XDY$ for some $X,Y\in\mathrm{GL}(\mathbb{Z}_p)$, it always exists and is almost unique, see~\cite{Smith} ). There are the well known explicit formulas computing the diagonal entries of $D$, namely $D_{i,i}=\frac{d_i(A)}{d_{i-1}(A)}$ where $d_i$ is the greatest common divisor of all $i\times i$ minors of $A$. We can think of $m_i^{d-i}$'s as generalizations of these formulas (Remark~\ref{SmithRem}).

The Section~\ref{sec7} is a natural generalization of the foregoing section: we ask whether some $\mathbb{Z}$-linear combination of $m_i^{d-i}$'s computed for three or more matrices has a geometric meaning. The situation become much more challenging, since in the case of two vertices the axiom~\eqref{ax2} guarantees that we can choose (after change of base) diagonal matrix representatives of the vertices. Having diagonal representatives, to compute the $m_i^{d-i}$'s is rather easy. But the case when axiom~\eqref{ax2} holds for three vertices is uncommon. Our main result provides one such non-trivial formula, which gives (lower bound of) the smallest total distance of a vertex from a given finite set of vertices:
\begin{thm}
Given any collection of vertices $\alpha_1,\ldots,\alpha_n\in\mathscr{B}(\SL_d(\mathbb{Q}_p)$ we have
\begin{equation}\label{mainResult}
\min_{\beta\in\mathscr{B}^0} \bigg\lbrace\sum_{i=1}^n\delta(\alpha_i,\beta)\bigg\rbrace\geq \max_{\pi\in S_n} \bigg\lbrace m_{A_{i}:d}^{A_{\pi(i)}:0}-m_{A_{i}:d-1}^{A_{\pi(i)}:1}\bigg\rbrace.
\end{equation}
Equality in~\eqref{mainResult} holds whenever all $\alpha_i$'s lie on a common apartment. 
\end{thm}
In the proof, an essential trick is a use of the well known Hall's marriage theorem. The author strongly believe the equality holds in the general case as well, he consider it would be better to leave results in this direction (together with some other results) to a further article because of the length of the proofs. 

We also observe (in the case of one apartment) that every local minimum of the left hand side of~\eqref{mainResult} is already a global minimum. In the case $n=3$ the Theorem~\ref{mainResult} gives us a formula (lower bound) of the length of a minimal tree connecting three given vertices.

In the appendix we consider the case of $\SL_2(\mathbb{Q}_p)$ and give another formula of different kind, counting number of edges shared by two given apartments.

\section{Preliminaries}
Through the whole article we assume the knowledge of basic facts about the field of p-adic numbers $\mathbb{Q}_p$, which we will use without reference. We give a very quick overview
to p-adic numbers here.
We can think of the elements of $\mathbb{Q}_p$ as the infinite power series: Each $0\not=a\in\mathbb{Q}_p$ can be uniquely written in the form
\begin{equation}\label{eqPadic}
a=\sum_{i=n}^\infty a_i\cdot p^i,
\end{equation}
where $n\in\mathbb{Z}$, $a_i$'s are taken from $\{0,\ldots,p-1\}$ and $a_{n}\not =0$.
The field is equipped with an additive valuation $\nu_p:\mathbb{Q}_p\rightarrow \mathbb{Z}$ defined by  
\begin{equation}\label{eqDefOfVal}
\nu_p(a)=\begin{cases}
      \hfill \max\{k\in\mathbb{Z}:a\in(p)^k\}\hfill &\text{if\ \ }\hfill a\not= 0, \\
      \hfill \infty \hfill &\text{if\ \ }\hfill   a=0. \\
\end{cases}
\end{equation}
For every $a_1,\ldots,a_n\in\mathbb{Q}_p$ the valuation $\nu_p$ satisfies these two conditions:
\begin{align*}
\nu_p(a_1\cdot\cdots\cdot a_n)&=\nu_p(a_1)+\cdots+\nu_p(a_n),\\
\nu_p(a_1+\cdots+a_n)&\geq \min\{\nu_p(a_1),\ldots,\nu_p(a_n)\}.
\end{align*}
We call the second one \textit{triangular inequality}.
The ring of $p$-adic integers $\mathbb{Z}_p\subseteq \mathbb{Q}_p$ consists of the elements of $\mathbb{Q}_p$ with non-negative valuation.  
$\mathbb{Z}_p$ is a local ring with the maximal ideal $(p)$. The invertible elements in $\mathbb{Z}_p$ are exactly those with zero valuation.

Next define others basic algebraic concepts, used in the article:
By $\mathbb{Z}_p$-lattice we mean a $\mathbb{Z}_p$-submodule $\mathcal{L}$ of some $\mathbb{Q}_p$-vector space $\mathcal{V}$, such that $\mathcal{L}$ is finitely generated and generate whole $\mathcal{V}$. A key fact is that every $\mathbb{Z}_p$-lattice $\mathcal{L}$ has $\mathbb{Z}_p$-basis. This is consequence of the well-known Theorem of elementary divisors.

Let $R$ be a ring, then by $GL_d(R)$ we denote the group of all invertible $(d\times d)$-matrices over $R$. For example, $GL_d(\mathbb{Z}_p)$ are those matrices having determinant in $\mathbb{Z}_p^\times$. By $\SL_d(R)$ we denote all $(d\times d)$-matrices over ring $R$ having determinant equal to $1$. We will also use group of diagonal $(d\times d)$-matrices denoted by $\mathrm{Diag}_d(R)$. Given matrix $A$ we will refer to its element at the position $(i,j)$ by $A_{i,j}$.

We also assume the knowledge of the concept of abstract simplicial complexes. Given any abstract simplicial complex $\Delta$, by $\Delta^i$, $i\in\mathbb{N}^0$, we denote it's $i$-dimensional substructure. Hence $\Delta^1$ is essentially a graph and we will treat it as a graph. Finally, we will arrive at a slight distinction between a walk and a path in a graph, the second is assumed to avoid repetition in its vertices.
\section{Construction of the Bruhat-Tits building for group \texorpdfstring{$\SL_d(\mathbb{Q}_p)$}{SLd(Qp)}}\label{sec3}
In the section we describe a standard construction of the Bruhat-Tits building for group $\SL_d(\mathbb{Q}_p)$. We follow the Garrett's book~\citep{Gar}, section 19.1. We will not verify the axioms~(\ref{ax1}--\ref{ax3}), for this see the section 19.2 in~\citep{Gar}. More precisely stated axiom~\eqref{ax1} could be found in the book as well.

Through the whole article $p$ and $d$ will always refer to some fixed prime and some integer $d \geq 2$. Hence, we can freely leave reference to $p$ and $d$ from our notion. We will use $\nu$ instead of $\nu_p$ and by $\mathscr{B}$ we denote building for $\mathrm{SL}_d(\mathbb{Q}_p)$, which we are going to define.

By the well-known Theorem of elementary divisors all torsion-free modules over a discrete valuation ring are free (see~\cite{Lang}, Theorem 7.3.). Hence, each $\mathbb{Z}_p$-lattice $\mathcal{L}\subset \mathbb{Q}^d_p$ has some $\mathbb{Z}_p$-basis.

\begin{defn}
The building of the group $\SL_d(\mathbb{Q}_p)$, denoted by $\mathscr{B}\equiv\mathscr{B}(\SL_d(\mathbb{Q}_p))$, is an abstract simplicial complex constructed as follows: for the set of vertices we take
$$\{\mathcal{L}:\mathcal{L}\text{\ is\ }
\mathbb{Z}_p-\text{lattice}\text{\ in\ }\mathbb{Q}_p^d\}
/\sim,$$
where $\sim$ is an equivalence which identifies $\mathcal{L}\sim a\cdot \mathcal{L}$, for each $\mathbb{Z}_p$-lattice $\mathcal{L}$ and $a\in\mathbb{Q}_p^\times$. We call lattices equivalent in this fashion  homothetic. And for each $(k+1)$-tuple of $\mathbb{Z}_p$-lattices
\begin{equation}\label{eqChain}
\mathcal{L}_0\supsetneq \mathcal{L}_1\supsetneq \cdots\supsetneq \mathcal{L}_k\supsetneq p\mathcal{L}_0
\end{equation}
we add the $k$-simplex, whose vertices are that given by $\mathbb{Z}_p$-lattices $\mathcal{L}_i$'s. 
\end{defn}
Although for $d\geq 3$ it is hard to visualize the building, the local picture is quite simple to describe: the $k$-simplices containing $[\mathcal{L}_0]$ are in one-to-one correspondence with chains of $\mathbb{F}_p$-vector spaces $\mathbb{F}_p^d\gneq \mathcal{V}_1\gneq\cdots \gneq \mathcal{V}_k\gneq 0$. Just factorize~\eqref{eqChain} by $p\mathcal{L}_0$ and identify $\mathbb{Z}/p\mathbb{Z}_p\equiv \mathbb{F}_p$. Easily follows:
\begin{lemma}\label{lemDimOfB}
Every maximal simplex of $\mathscr{B}$ is of dimension $d-1$. Specially, the whole building has dimension $d-1$.
\end{lemma}
\begin{rem}\label{defAction}
The action of $\SL_d(\mathbb{Q}_p)$ on $\mathscr{B}$ is that induced by the standard action on the vector space $\mathbb{Q}_p^d$. It is easy to see that $\SL_d(\mathbb{Q}_p)$ acts by simplicial automorphisms.
\end{rem}
\begin{defn}
For each basis $\mathcal{F}=\{u_1,\ldots,u_d\}$ of $\mathbb{Q}_p^d$ call an apartment the simplicial sub-simplex $\mathscr{A}$ generated by the vertices having following lattice representatives
$$\mathbb{Z}_p[p^{e_1}\cdot u_1,\ldots,p^{e_d}\cdot u_d],$$
where $e_i's$ are integers.  We call $\mathcal{F}$ a splitting basis for the apartment $\mathscr{A}$.
\end{defn}
Note that Theorem of elementary divisors assures a weak version of the axiom~\eqref{ax2} for pair of vertices.

\section{Matrix representatives}\label{sec4}

In computations it will be useful to have in hand some matrix data coding the vertices and the apartments. For this purpose we introduce \textit{matrix representatives}. Working with matrix representatives we always assume we have some basis $\mathcal{F}$ of $\mathbb{Q}_p^d$. The geometric meaning of all coordinates data is always related to the basis $\mathcal{F}$. Moreover, denote by $\mathcal{E}$ the canonical basis of $\mathbb{Q}_p^d$.

\begin{defn}\label{defRep}
Let $A\in GL_d(\mathbb{Q}_p)$. Having some basis $\mathcal{F}=\{u_1,\ldots,u_d\}$ of $\mathbb{Q}_p^d$, denote by $A_{v:\mathcal{F}}$ the vertex defined by the $\mathbb{Z}_{p}$-lattice spanned by 
\begin{equation}\label{eqSpan}
v_i=A_{i,1}\cdot u_1+A_{i,2}\cdot u_2+\cdots +A_{i,d}\cdot u_d,\ \ \ i=1,\ldots,d.
\end{equation}
And by $A_{a:\mathcal{F}}$ we denote the apartment, whose splitting basis is that from~\eqref{eqSpan}. That is, the vertices of $A_{a:\mathcal{F}}$ are exactly 
$$\{(DA)_{v:\mathcal{F}}: D\in \mathrm{Diag}_d(\mathbb{Q}_p)^\times\}.$$
\end{defn}

The following lemma concerns the ambiguity of just defined matrix representatives of vertices.
\begin{lemma}\label{lemUniquenesOfRep}
Let $A,B\in \GL_d(\mathbb{Q}_p)$ and $\mathcal{F}$ be a basis of $\mathbb{Q}^d$, then $A_{v:\mathcal{F}}=B_{v:\mathcal{F}}$ if and only if there are $X\in \GL_d(\mathbb{Z}_p)$ (i.e., $\det(X)\in\mathbb{Z}_p^\times$) and $i\in\mathbb{Z}$ such that 
$$B=p^i XA.$$
\end{lemma}
\begin{proof}
The matrix $X$ corresponds to the ambiguity in choosing $\mathbb{Z}_p$-basis of $\mathrm{Row}_{\mathbb{Z}_p}(A)$ and power of $p$ treats the ambiguity in choosing a concrete representative of the homothety class.
\end{proof}


Working with coordinates, a common trick is choosing a suitable basis $\mathcal{F}$. For example, investigating an apartment $\mathscr{A}$ we choose basis $\mathcal{F}$ to be a splitting basis for $\mathscr{A}$, hence $\mathscr{A}=I_{a:\mathcal{F}}$. 
We have to describe how the skipping between bases affects the matrix representatives.
\begin{lemma}\label{lemChangeOfBasis}
Let $\mathcal{F}_1$ and $\mathcal{F}_2$ be two bases of $\mathbb{Q}^d_p$ and $T$ be the matrix of transition between them (i.e., rows of $T$ are coordinates of vectors of $\mathcal{F}_1$ with respect to $\mathcal{F}_2$). For any $A\in GL_d(\mathbb{Q}_p)$ we have
$$A_{v:\mathcal{F}_1}=(AT)_{v:\mathcal{F}_2}\text{\ \ and\ \ }A_{a:\mathcal{F}_1}=(AT)_{a:\mathcal{F}_2}.$$
\end{lemma}
\begin{proof}
Proving this is just an easy exercise from linear algebra.
\end{proof}
In most of cases it will be clear with respect to which basis we use the matrix representatives. Hence, we will usually leave the reference to basis from the notion (e.g., vertices will be defined as $A_v$).   
\begin{rem}\label{fredomInBase}
Notice that matrix representatives of vertices are unique up to multiplication by certain matrices from left and skipping between two bases $\mathcal{F}_1$ and $\mathcal{F}_2$ affects the matrix representatives by multiplication by the transition matrix from right. Hence, from associativity of matrix multiplication follows:
$$A_{v:\mathcal{F}_1} =B_{v:\mathcal{F}_1}\Leftrightarrow A_{v:\mathcal{F}_2} =B_{v:\mathcal{F}_2}.$$
\end{rem}
\begin{defn}\label{defRelsimVA}
Define a relation of equivalence $\sim_v$ on the set of matrices in $\GL_{d}(\mathbb{Q}_p)$ by
$ A\sim_v B$ iff $A_{v:\mathcal{F}}=B_{v:\mathcal{F}}$ for some (equivalently for all) basis $\mathcal{F}$.
\end{defn}
The essential knowledge we gain from the following lemma, is that, never mind which basis we choose, the $\GL_d(\mathbb{Q}_p)$ acts on the matrix representatives by matrix multiplication from right. 
\begin{lemma}\label{lemActInRep}
Let $\rho$ be the action from Definition~\ref{defAction}. Given $g\in \GL_d(\mathbb{Q}_p)$, with respect to the canonic basis $\mathcal{E}$ the action of $g$ on the vertices of $\mathscr{B}$ is this:
$$\rho_g:A_{v:\mathcal{E}}\mapsto (Ag^{-1})_{v:\mathcal{E}}.$$
If basis $\mathcal{F}$ arises from $\mathcal{E}$ by transition matrix $T$ (i.e., $\mathcal{F}$ is given by rows of $T$) then $g$ acts on the vertices of $\mathscr{B}$ as follows:
$$\rho_g:A_{v:\mathcal{F}}\mapsto (AT^{-1}g^{-1}T)_{v:\mathcal{F}}.$$
\end{lemma}
\begin{proof}
An easy exercise from linear algebra.
\end{proof}
\begin{lemma}\label{lemHelpInIntro}
Vertices $\alpha_0,\ldots,\alpha_k$ form an
 $k$-simplex if and only if there are matrices $A_0,\ldots,A_k\in\GL_d(\mathbb{Q}_p)$, such that $(A_i)_v$'s are all the vertices $\alpha_i$'s and
\begin{equation}\label{eqSimInCoor}
\mathrm{Row}_{\mathbb{Z}_p}(A_0)\supsetneq\mathrm{Row}_{\mathbb{Z}_p}(A_1)\supsetneq\cdots\supsetneq\mathrm{Row}_{\mathbb{Z}_p}(A_{k})\supsetneq p\mathrm{Row}_{\mathbb{Z}_p}(A_0).
\end{equation}
\end{lemma}
\begin{proof}
There has to be a chain of lattices like in~\eqref{eqChain} witnessing the $k$-simplex. We take for the matrices $A_i$'s the matrices whose rows generates the lattices in the sense of Definition~\ref{defRep}. 
\end{proof}

\section{Localized form and distance-formulas}\label{sec5}
In the section we introduce a concept of distance formulas, which is key to the article.
In most of the definitions and theorems we will work with matrix representatives of vertices and apartments, so any matrix occurring in this and further sections, which is not specified, is intended to be an element of $\mathrm{GL}_d(\mathbb{Q}_p)$. We will work in coordinates, hence all coordinate-like data will be expressed with respect to some basis $\mathcal{F}$. If there will be no danger of confusion, we will leave the reference to $\mathcal{F}$ from the notion of matrix representatives. Since the choice of $\mathcal{F}$ is arbitrary we will often chose the one suitable to the situation.

For each square matrix $X$ over $\mathbb{Q}_p$ define $$|X|_p=\nu(\det(X)).$$
\begin{defn}
Let $A$ and $B$ be two matrices from $\GL_d(\mathbb{Q}_p)$ and $i$ an integer, $0\leq i\leq d$. Call $\text{\footnotesize{(A:i,B:d-i)}}$-\textit{candidate} any matrix $X$ composed from $i$ rows of $A$ and $d-i$ rows of $B$. Define
$$
m_{A:i,B:d-i}:=\min\{|X|_p: X\text{ is \footnotesize{(A:i,B:d-i)}}\text{-candidate}\}.
$$
Call $\text{\footnotesize{(A:i,B:d-i)}}$-\textit{witness} any $\text{\footnotesize{(A:i,B:d-i)}}$-candidate $X$ having $|X|_p=m_{A:i,B:d-i}$.
\end{defn}   
In the following lemma we will need this fact: each matrix from $\GL_d(\mathbb{Z}_p)$ can be decomposed into a product of elementary matrices with integral elements. It is proved for any local ring in~\cite{ARD}, Corollary 13.1.4.
\begin{lemma}\label{lemIn1:1}
Let $A, B\in \GL_{d}(\mathbb{Q}_p)$ and $k\in\mathbb{Z}$, $0\leq k\leq d$. The integer $m_{A:k,B:d-k}$ is invariant under left multiplication of $A$ or $B$ by matrices in $\GL_d(\mathbb{Z}_p)$. That is
$$m_{A:k,(XB):d-k}=m_{A:k,B:d-k}=m_{(XA):k,B:d-k},\text{\ \ for\ each\ }X\in \GL_d(\mathbb{Z}_p).$$
\end{lemma}
\begin{proof}
We prove that $m_{A:k,B:d-k}$ is invariant under multiplying $A$ by the elementary matrices (i.e., those which correspond to elementary row transformations). Directly from definition, permuting rows and multiplying any row by an element of $\mathbb{Z}_p^\times$  does not change 
$m_{A:k,B:d-k}$. Let $A'$ be the matrix with the same rows as $A$ except $A'_{i,*} = A_{i,*}+r\cdot A_{j,*}$, where $r\in \mathbb{Z}_p$ and $i\neq j$. 
Choose any {\textit{\footnotesize{($A'$:k,B:d-k)}}}-witness $X$. It suffices to prove $m_{A':k,B:d-k}\geq m_{A:k,B:d-k}$ (since the elementary row transformations are invertible). 

If $X$ does not contain row $A'_{i,*}$, it is also {\textit{\footnotesize{(A:i,B:d-i)}}}-candidate, and hence the inequality is trivial.
Assume that $X$ contains $A'_{i,*}$. According to the well known rules for computing determinant we have
$$\det(X)=\det(X_i)+r\cdot \det(X_j),$$
where we get $X_i$ ($X_j$, resp.) by replacing the row $A'_{i,*}$ in $X$ by $A_{i,*}$ ($A_{j,*}$, resp.). Both $X_i$ and $X_j$ are {\textit{\footnotesize{(A:k,B:d-k)}}}-candidates, so 
$$m_{A':k,B:d-k} \geq \min\{|X_i|_p, |X_j|_p+
\nu(r)\}
\geq m_{A:k,B:d-k}.$$
First inequality follows from the triangular inequality and the second one from definition of $m_{A:k,B:d-k}$ (note that $\nu(r)\geq 0$, since $r\in\mathbb{Z}_p$).
\end{proof}

Now, using $m_{A:i,B:d-i}$'s, we can consider expressions in the form
\begin{equation}\label{generalForm}
\sum_{i=0}^{d}a_{i}\cdot m_{A:i,B:d-i},
\end{equation}
where $a_i$'s are  elements of $\mathbb{Z}$, and investigate when they have some geometric meanings. 

\begin{prop}\label{invarianceCon}
Let $a_0,\ldots,a_d\in \mathbb{Z}$, for each $A,B\in \GL_{d}(\mathbb{Q}_p)$ define 

\begin{equation}\label{invExp}
\gamma(A,B)=\sum_{i=0}^{d}a_{i}\cdot m_{A:i,B:d-i}.
\end{equation}
Fix concrete pair $A,B$. The following two conditions are equivalent:
\begin{enumerate}[(i)]
\item there is $\gamma(A,B)=\gamma(A',B')$ for any $A'\sim_v A$ and $B'\sim_v B$,
\item the following equations hold:
\begin{equation}\label{invariantCon1}
\sum_{i} i\cdot a_i=0,
\end{equation}
\begin{equation}\label{invariantCon2}
\sum_{i} (d-i)\cdot a_i=0.
\end{equation}
\end{enumerate}
Furthermore, these three conditions are equivalent:
\begin{enumerate}[(i)]
\setcounter{enumi}{2}
\item\label{con4}$\gamma(A,B)=\gamma(AT,BT)$ for each $T\in \GL_d(\mathbb{Q}_p)$,
\item\label{con5} $\gamma(A,B)=\gamma(AT,BT)$ for some $T\in \GL_d(\mathbb{Q}_p)$, with $|T|_p\not=0$,
\item\label{con6} following equation holds
\begin{equation}\label{invariantCon3}
\sum_{i} a_i=0.
\end{equation}
\end{enumerate}
If any two of the equations~\eqref{invariantCon1},~\eqref{invariantCon2} and~\eqref{invariantCon3} hold, all of them hold and~\eqref{invExp} well-defines function
$\gamma:\mathscr{B}^0\times \mathscr{B}^0\rightarrow\mathbb{Z}$.
\end{prop}
\begin{proof}
According to 
Definition~\ref{defRelsimVA}
we can go through all $A'$ with $A'\sim_v A$ by multiplying $A$ from left by $r\cdot Q$, where $Q\in \GL_d(\mathbb{Z}_p)$ and $r\in \mathbb{Q}_p^\times$. Lemma~\ref{lemIn1:1} states that multiplying by $Q$ causes no change: $\gamma(A,B)=\gamma(Q\cdot A,B)$. 
Next, every {\textit{\footnotesize{(rA:i,B:d-i)}}}-candidate
$Y$ is just {\textit{\footnotesize{(A:i,B:d-i)}}}-candidate $X$ with $i$ rows multiplied by $r$, so $|Y|_p=|X|_p+i\cdot\nu(r)$.
 It easily follows 
 $$m_{(r\cdot A):i,B:d-i} = m_{A:i,B:d-i}+i\cdot\nu(r),$$ hence
$$\gamma(r\cdot Q\cdot A,B)=\gamma(A,B)+\nu(r)\cdot\sum_{i}i\cdot a_i.$$ Obviously, the requirement $\gamma(A,B)=\gamma(A',B)$ for each $A'\sim_v A$ is equivalent to equation~(\ref{invariantCon1}). Analogous investigation for $B$ leads to equation~(\ref{invariantCon2}). 
Since the equation~(\ref{invariantCon2}) does not depend on $A$,  
if equations~\eqref{invariantCon1} and~\eqref{invariantCon2} hold, then $\gamma(A,B)=\gamma(A',B)=\gamma(A',B')$ for any $A'\sim_v A$ and $B'\sim_v B$.

The {\textit{\footnotesize{(AT:i,BT:d-i)}}}-candidates are exactly the $XT$'s where $X$'s are {\textit{\footnotesize{(A:i,B:d-i)}}}-candidates. 
Hence
$$m_{AT:i,BT:d-i}= \min\{|XT|_p:X\text{\textit{ is \footnotesize{(A:i,B:d-i)}}\text{-candidate}}\}=$$
$$=\min\{|X|_p:X\text{\textit{ is \footnotesize{(A:i,B:d-i)}}\text{-candidate}}\}+|T|_p=m_{A:i,B:d-i}+|T|_p.$$
It follows $$\gamma(AT,BT)=\gamma(A,B)+|T|_p\sum_{i}a_i.$$ 
Using this formula, all equivalences $\eqref{con4}\Leftrightarrow\eqref{con5} \Leftrightarrow \eqref{con6}$ are quite clear.

Finaly the last part: The irrelevance of one of the conditions is easy to see. According to Lemma~\ref{lemUniquenesOfRep} the conditions~\eqref{invariantCon1} and~\eqref{invariantCon2} assure the $\gamma$ is well defined on vertices. The Lemma~\ref{lemChangeOfBasis} states that change of basis $\mathcal{F}$ effects the matrix representatives by right multiplication by certain matrix $T$, this is treated by condition~\eqref{invariantCon3}.
\end{proof}
We will call the equations~\eqref{invariantCon1}, \eqref{invariantCon2} and~\eqref{invariantCon3} from the previous lemma \textit{invariance conditions}. And we call any formula~\eqref{generalForm} satisfying invariance conditions \textit{distance formula}. Let us extend the notion:
$$m_{A:i,B:d-i}\equiv m_{A:i}^{B:d-i}\equiv m_i^{d-i}.$$
The economical one on the right will be used when the pair $(A,B)$ is fixed.

There is a natural way how to using elementary row $\mathbb{Z}_p$-operations transform $A$ into $A'$ being in so called \textit{localized form}, where  we can easily compute $m_{A':i,I:d-i}$ (hence $m_{A:i,I:d-i}$ as well by Lemma~\ref{lemIn1:1}). The term \textit{localized form} is not standard, it is used only in this article.
\begin{defn}\label{localForm}
We say that a matrix $A\in \mathrm{GL}_d(\mathbb{Q}_p)$ is in localized form if there is a permutation $\pi\in S_d$, so that for each $i=1,\ldots,d$, we have:
\begin{enumerate}[(i)]
\item\label{loc1} $A_{i,\pi(i)}$ is a power of $p$ and $\nu(A_{i,\pi(i)})$ is minimal among $\nu(A_{j,k})$, $i\leq j\leq d,1\leq k\leq d$,
\item\label{loc2} $A_{j,\pi(i)}=0$ for all $j>i$,
\item\label{loc3} $\nu(A_{i,j}) > \nu(A_{i,\pi(i)})$ for all $j< \pi(i)$.
\end{enumerate}
\end{defn}
\begin{example}\label{exOfLocalForm}
Following matrix is in localized form with $\pi=(1,2,3)$:
$$A=\left( \begin{array}{ccc}
    p^2 & \mathbf{p} & p\\
    p^4 & 0 & \mathbf{p^3}\\
    \mathbf{p^{5}} & 0 & 0\\
    \end{array} \right).
$$
The elements on positions $(i,\pi(i))$, $i=1,2,3$ are stressed by bold font. Note that $|A|_p=1+3+5$.
\end{example}
\begin{lemma}\label{localizedRepresentative}
For each matrix $A\in \mathrm{GL}_d(\mathbb{Q}_p)$ there is a matrix $A'$ in localized form and $X\in \GL_d(\mathbb{Z}_p)$, such that $A'=XA$.
\end{lemma}
\begin{proof}
We will transform $A$ to $A'$ using elementary row transformations. The algorithm is similar to the classical one for making a matrix up triangular. 

We will rewrite $A$ in $d-1$ steps.
At $i$-th step (for $i\geq 1$) we find a row $A_{j,*}$, $j\geq i$, which contains an element $A_{j,k}$ so that $\nu(A_{j,k}) 
\leq \nu(A_{l,m})$ for all $l\geq i$ and all $m$; we may have freedom in choosing $A_{j,k}$, if we 
chose in row $A_{j,*}$ the one with minimal $k$, the last condition~\eqref{loc3} will hold. Now we can make zero each $A_{l,k}$, $j\neq l\geq i$ by adding $(-A_{l,k}/A_{j,k})\cdot A_{j,*}$ to $A_{l,*}$. Note that $-A_{l,k}/A_{j,k}$ is an element of $\mathbb{Z}_p$, since we have chosen $A_{j,k}$ so that $\nu(A_{j,k})\leq \nu(A_{l,k})$, hence the operations cannot decrease valuation of any element on effected position under $\nu(A_{j,k})$. Finally, we interject rows $A_{i,*}$ and $A_{j,*}$ (there will be $\pi(i)=k$).

At the end we multiply each row by an invertible element in $\mathbb{Z}_p$ to make all $A_{i,\pi(i)}$'s powers of $p$.
\end{proof}
\begin{lemma}\label{eq2.2}
Let $A$ be in localized form with $\alpha_i=|A_{i,\pi(i)}|_p$, where $\pi$ is the permutation from Definition~\ref{localForm}. Then for each $i=0,\ldots,d$ we have
\begin{equation}\label{eq1.1}
m_{A:i,I:d-i}=\sum_{j=1}^{i}\alpha_j.
\end{equation}
\end{lemma} 
\begin{proof}
Let $X$ be some {\textit{\footnotesize{(A:i,I:d-i)}}}-candidate. Every summand $s$ in 
$$\det(X)=\sum_{\rho\in S_d}\mathrm{sgn}(\rho)\cdot X_{1,\rho(1)}\cdot \cdots \cdot X_{d,\rho(d)}$$ 
is a product of $i$ elements $A_{i,j}$'s and some $1$'s and $0$'s (entries of $I$), hence 
$$\nu(s)\geq \nu(a_{1,\pi(1)}
\cdot\cdots\cdot a_{i,\pi(i)})=\sum_{j=1}^{i}\alpha_j,$$
 that follows from condition~\eqref{loc1} 
in Definition~\ref{localForm}. By the triangular inequality $|X|_p\geq \sum_{j=1}^{i}\alpha_j$. On the 
other hand, if we take the first $i$ rows of $A$ and $d-i$ rows of $I$ having $1$'s not at positions $\pi(1),\ldots,\pi(i)$, then we get a {\textit{\footnotesize{(A:i,I:d-i)}}}-candidate $X$, whose determinant has by condition~\eqref{loc2} in Definition~\ref{localForm} only one non-vanishing summand $a_{1,\pi(1)}\cdot\cdots\cdot a_{i,\pi(i)}$ (respective $1$ if $i=0$).
\end{proof}
Next corollary provides an insight in what the $m_{i}^{d-i}$'s could possibly be - they form concave sequence. That is, having some fixed $A, B\in \GL_d(K)$, for all integers $0\leq j< k<  l \leq d$ we have 
\begin{align}\label{concavity}
\frac{1}{k-j}\cdot (m_{k}^{d-k}-m_{j}^{d-j})\leq \frac{1}{l-k}\cdot(m_{l}^{d-l}-m_{k}^{d-k}).
\end{align}
That directly follows from inequality \eqref{firstEq}, which we are going to prove.
\begin{cor}[concavity]
Let $A,B\in \GL_d(K)$ and $i,j,k$ be integers such that $0\leq j\leq k\leq  l \leq d$. Then
following inequality holds and satisfies invariance conditions 
\begin{align}\label{firstEq}
\begin{aligned}
0&\leq (k-j)\cdot(m_{l}^{d-l}-m_{k}^{d-k})+(l-k)\cdot (m_{j}^{d-j}-m_{k}^{d-k})=\\
&=(k-j)\cdot m_{l}^{d-l}+(j-l)\cdot m_{k}^{d-k}+(l-k)\cdot m_{j}^{d-j}.
\end{aligned}
\end{align}
Where $m_{i}^{d-i}=m_{A:i,B:d-i}$.
\end{cor}
\begin{proof}
The verification of invariance conditions is straightforward as well as proving the equality. Because of the invariance conditions, the value of the formula does not change if we exchange $A,B$ as follows: $(A,B)\rightsquigarrow (AB^{-1},I)\rightsquigarrow (A',I)$, where matrix $A'\sim_v AB^{-1}$ is in localized form. The existence of $A'$ is ensured by Lemma~\ref{localizedRepresentative}. Now let us apply Lemma~\ref{eq2.2} for $A'$:

$$(k-j)\cdot(m_{A':l,I:d-l}-m_{A':k,I:d-k})+(l-k)\cdot (m_{A':j,I:d-j}-m_{A':k,I:d-k})=$$
$$=\left((k-j)\cdot \left(\sum_{i=1}^{l}\alpha_i-\sum_{i=1}^{k}\alpha_i\right)\right)-\left((l-k)\cdot\left(\sum_{i=1}^{k}\alpha_{i} -\sum_{i=1}^{j}\alpha_{j}\right)\right)=$$
$$=\left((k-j)\cdot \sum_{i=k+1}^{l}\alpha_i\right)-\left((l-k)\cdot\sum_{i=j+1}^{k}\alpha_{i} \right).$$
The $\alpha_j$'s are defined in Lemma~\ref{eq2.2}. In both parentheses is a sum of $(k-j)\times (l-k)$ $\alpha_i$'s and each $\alpha_i$ in the first parentheses is greater than any $\alpha_i$ in the second one. Hence the inequality holds.
\end{proof}

\section{The relative position of two vertices}\label{sec6}
In the section we introduce some formulas describing relative positions of two given vertices. The strategy will be as follows: having two vertices $\alpha$ and $\beta$ we choose a basis $\mathcal{F}$ so that $\beta=I_{v:\mathcal{F}}$ and then we find a matrix $A$ in localized form such that $\alpha=A_{v:\mathcal{F}}$. In this situation we can compute $m_{A:i,I:d-i}$ and investigate if there is a geometric meaning of some given formula satisfying invariance conditions.

Once we have a matrix $A$ in localized form (with permutation $\pi$), we can easily construct an apartment containing $A_v$ and $I_v$ and a walk connecting them. We show this fact on an example, but the construction works in the general case, hence we will formulate some observations in general setting. Express the matrix $A$ in localized form from Example~\ref{exOfLocalForm} as a product $DA'$, $A'\in\mathrm{GL}_d(\mathbb{Z}_p)$:
$$
A=\left( \begin{array}{ccc}
    p^2 & p & p\\
    p^4 & 0 & p^3\\
    p^{5} & 0 & 0\\
    \end{array} \right)=\left( \begin{array}{ccc}
    p & 0 & 0\\
    0 & p^3 & 0\\
    0 & 0 & p^5\\
    \end{array} \right)\cdot
    \left( \begin{array}{ccc}
    p & 1 & 1\\
    p & 0 & 1\\
    1 & 0 & 0\\
    \end{array} \right).
$$
We have $\alpha_1=1,\alpha_2=3,\alpha_3=5$ and $|A|_p=6$. 
See that $A'_v=I_v$ and $A_v\in A'_a$, hence both vertices $I_v$ and $A_v$ lie in the apartment $A'_a$. Because we can make $I_v$ an arbitrary vertex (by choosing suitable basis $\mathcal{F}$), we observe a weak version of axiom~\eqref{ax2}:
\begin{obs}\label{obs0}
Every two vertices have some common apartment.
\end{obs}
Furthermore, in the apartment $A'_a$ we have following walk from $I_v$ to $A_v$:
$$I_v=\big(\mathrm{diag}(1,1,1)\cdot A'\big)_v=\big(\mathrm{diag}(p,p,p)\cdot A'\big)_v\leadsto\big(\mathrm{diag}(p,p^2,p^2)\cdot A'\big)_v\leadsto$$
$$\rightsquigarrow\big(\mathrm{diag}(p,p^3,p^3)\cdot A'\big)_v\leadsto\big(\mathrm{diag}(p,p^3,p^4)\cdot A'\big)_v\leadsto\big(\mathrm{diag}(p,p^3,p^5)\cdot A'\big)_v=A_v.$$
Notice that the length of the walk is $\alpha_3-\alpha_1=5-1=4$. 

\begin{obs}\label{obs1}
If $A$ is in localized form and $\alpha_i$'s are as in Definition~\ref{localizedRepresentative}, then there is a walk from $I_v$ to $A_v$ of length $\alpha_d-\alpha_1$.
\end{obs}
We will later observe (in Remark~\ref{remTheWalkIsPath}) that the walk is a path (it has no repetition in edges neither in vertices). The fact that we can ``localize" the vertex $A'_v$, with respect to the origin $I_v$, is the motivation for the term \textit{localized form}.

\begin{obs}\label{obs2}
If vertices $\alpha$ and $\beta$ are connected by some edge and $\mathcal{F}$ is a basis, we can choose matrix representatives $\alpha=A_{v:\mathcal{F}}$ and $\beta=B_{v:\mathcal{F}}$ such that $B$ comes from $A$ by multiplying some (but not all) rows by $p$.
\end{obs}
\begin{proof}
If the edge belongs to the apartment $I_a$, there are diagonal representatives $A,B$ such that 
$\mathrm{Row}_{\mathbb{Z}_p}(A)\supset \mathrm{Row}_{\mathbb{Z}_p}(B) \supset p\cdot \mathrm{Row}_{\mathbb{Z}_p}(A)$. After multiplying rows of $B$ by convenient units of $\mathbb{Z}_p$, we are done. In the general case we can use Observation~\ref{obs0} to arrive at the foregoing case. Skipping between bases corresponds to multiplying matrix representatives by some matrix from right (Lemma~\ref{lemChangeOfBasis}), which does not affect the desired relation between $A$ and $B$.
\end{proof}
Now we are ready to prove one of our main results: Theorem~\ref{distance}.
\begin{proof}[Proof of Theorem \ref{distance}]
Equation~(\ref{eqDistance}) respects the invariance conditions from Proposition~\ref{invarianceCon}, hence we can choose the basis $\mathcal{F}$ so that $\alpha$ and $\beta$ have matrix representatives $\alpha=A_v$ and $\beta=I_v$, where $A$ is in localized form. Using Lemma~\ref{eq2.2} we obtain $m_{0}^{d}-m_{1}^{d-1}= -\alpha_1$ and $m_{d}^{0}-m_{d-1}^{1}=\alpha_d$. By Observation~\ref{obs1} there is a walk of the length given by formula~\eqref{eqDistance}. It remains to prove there is no shorter path, which is an easy consequence of the following claim:

\textit{Claim:} The value of formula~\eqref{eqDistance} can change by replacing $\beta$ by some its neighbour $\beta'$ only by $\pm 1$.
\begin{proof}[Proof of the Claim]
By Observation~\ref{obs2} we can suppose $\beta=B_v$ and $\beta'=B'_v$, where $B'$ comes from $B$ by multiplying say $k<d$ rows by $p$. Choose any representative $A_v=\alpha$. Let us see how the summands of formula~\eqref{eqDistance} change under $B\rightsquigarrow B'$: $m_{A:d}^{B:0}=|A|_p$ stays the same and $-m_{A:d-1}^{B:1}$ could stay the same or decrease by $-1$ if the {\textit{\footnotesize{(A:d-1,C':1)}}}-witness contains one of the $k$ multiplied rows. Next, the number $m_{A:0}^{B:d}=|B|_p$ increases by $+k$ but the {\textit{\footnotesize{(A:1,B:d-1)}}}-witness contains at least $k-1$ multiplied rows, hence $-m_{A:1}^{C:d-1}$ decreases by $-k$ or $-(k-1)$, that is $m_{A:0}^{B:d}-m_{A:1}^{B:d-1}$ stays the same or increases by $+1$. Summing up, the whole formula can change only by $\pm 1$.
\end{proof}
\textit{finishing of the proof:}
If $A_v=B_v$ formula~\eqref{eqDistance} clearly equals zero. And if $A_v\not= B_v$, following any path from $\alpha$ to $\beta$ we need at least as many steps as the value of formula~\eqref{eqDistance}, because of the Claim.
\end{proof}
\begin{rem}\label{remTheWalkIsPath}
The value of the formula~\eqref{eqDistance} gives us the length of the walk from Observation~\ref{obs1} (as is clear form just finished proof), hence the walk is a path.
\end{rem}
\begin{rem}
The minimal path from Observation~\ref{obs1} 
lies on one apartment. Hence, there is no shortcut passing throuht several apartments, the obvious path from Observation~\ref{obs1} is the shortest one.
\end{rem}
\begin{example}
Notice that $$m_{A:d}^{B:0}-m_{A:0}^{B:d} \pmod{d}\equiv|A|_p-|B|_p \pmod{d}$$ satisfies invariance conditions modulo $d$. Fixing any vertex $\beta=B_v$ this formula well-defines mapping $\mathscr{B}^{0}\rightarrow |\mathbb{Z}_p|$ which extends to simplicial retraction $c_\beta$ of the whole building to the (external) $(d-1)$-simplex over elements of $\mathbb{Z}_d$. This simplicial retraction is called \textit{labeling} (For more details on labelings see~\cite{Gar}, section 4.4).
\end{example}
\begin{example}
Consider subgraph $\mathscr{B}\{1\}\subseteq\mathscr{B}^1$ where we take only edges $\{A_v,B_v\}$ with $|A|_p-|B|_p\pmod{d}\equiv\pm 1$. That is, for each edge we can take matrices representatives of the two vertices in such a way, that one arises from the other by multiplying one of its row by $p^{\pm 1}$. Distance in this graph equals
\begin{equation}\label{elDis}
m_{0}^{d}+m_{d}^{0}-m_{\lfloor \frac{d}{2}\rfloor}^{\lceil \frac{d}{2}\rceil}-m_{\lceil \frac{d}{2}\rceil}^{\lfloor \frac{d}{2}\rfloor}.
\end{equation}
The strategy of the proof follows the one of Theorem~\ref{distance}. Quick sketch: suppose $B=I$, $A$ is in localized form and moreover $\alpha_{\lceil \frac{d}{2}\rceil}=0$ (use the relation $A\sim_v p^j\cdot A$ for convenient $j$). Then there is in $\mathscr{B}\{1\}$ a path from $I_v$ to $A_v$ of the length 
$$\sum_{i=1}^d |\alpha_i|=\sum_{i=\lceil\frac{d}{2}\rceil+1}^{d}\alpha_i+\sum_{i=1}^{\lfloor \frac{d}{2}\rfloor}-\alpha_i
=\sum_{i=1}^{d}\alpha_i+0-\sum_{i=1}^{\lceil\frac{d}{2}\rceil}\alpha_i-\sum_{i=1}^{\lfloor \frac{d}{2}\rfloor}\alpha_i=\eqref{elDis}.$$
It remains to prove the variant of the Claim fitting in our situation.
\end{example}
\begin{example}\label{exDirectedGraph}
Consider even more restricted graph $\mathscr{B}[1]$, with directed edges $(B_v,A_v)$ such that $|A|_p-|B|_p\pmod{d}\equiv 1$. We see again that in the situation $B=I$, $A$ is in localized form with $\alpha_1=0$ we can get from $I_v$ to $A_v$ in $|A|_p$ steps. Since $0=\alpha_1=m^{A:1}_{I:d-1}$  and $m_{I:d}^{A:0}=|I|_p=0$ we have
\begin{equation}\label{directedDistance}
\overrightarrow{\delta}(I_v,A_v)\leq |A|_p = m^{A:d}_{I:0}-d\cdot m^{A:1}_{I:d-1}+(d-1)\cdot m^{A:0}_{I:d}.
\end{equation}
The last formula satisfies the invariance conditions and  similarly as in the case of Theorem~\ref{distance} we observe that each path from $B_v$ to $A_v$ has at least this length (hence equality holds in~\eqref{directedDistance}). 
\end{example}
\begin{rem}\label{SmithRem}
Well-known Principal divisors theorem says that every square matrix $A$ over PID ($\mathbb{Z}_p$ is PID) can be written as a product $XDY$ where $X,Y$ are invertible and $D$ diagonal. Moreover, diagonal entries of $D$ are unique up to permuting rows and multiplying rows by units. We also have explicit formulas of them: $D_{i,i}=\frac{d_i}{d_{i-1}}$ where $d_i$ is the greatest common divisor of all $i\times i$ minors of $A$. In our situation of $\mathbb{Z}_p$ we have only one prime ideal $(p)$, hence $d_i$ equals $p$ to the minimal valuation of $i\times i$ minors of $A$, that is $\nu(d_i)$ equals exactly $m_{A:i}^{I:d-i}$. 

We see that $\nu(D_{i,i})=m_{A:i}^{I:d-i}-m_{A:i-1}^{I:d-i+1}$ which we can calibrate into formula
$$\nu(D_{i,i})-\nu(D_{1,1})=m_{A:i}^{I:d-i}-m_{A:i-1}^{I:d-i+1}-m_{A:1}^{I:d-1}+m_{A:0}^{I:d},$$
which already satisfies invariance conditions.
\end{rem}

\begin{defn}\label{defRelCoordinates}
For any pair of vertices $\alpha=A_v$ and $\beta=B_v$ define the relative coordinates $\{\lambda_{i,\alpha:\beta}\}_{i=1}^{d}$ of $\alpha$ to the $\beta$ as:
$$\lambda_{i,\alpha:\beta}=m_{A:i,B:d-i}-m_{A:i-1,B:d-i+1}+m_{A:0,B:d} - m_{A:1,B:d-1}.$$
\end{defn}
\begin{lemma}\label{lemRelCoorAreSorted}
Let $\alpha=A_v$ and $\beta=B_v$ be any pair of vertices. Define $m_{k}^{d-k}=m_{A:k,B:d-k}$. For each $1\leq i <d$ we have
$$\lambda_{i+1,\alpha:\beta}-\lambda_{i,\alpha:\beta}=m_{i-1}^{d-i+1}-2m_{i}^{d-i}+m_{i+1}^{d-i-1}\geq 0.$$
Consequently, the relative coordinates of the $\alpha$ to the $\beta$ are sorted by size.
\end{lemma}
\begin{proof}
The equality follows directly from the definition of $\lambda_{i,\alpha:\beta}$'s and the inequality follows from inequality~\eqref{firstEq} for $0\leq j:=(i-1)\leq k:=i\leq l:=(i+1)\leq d$.
\end{proof}

If we are interested only in vertices in one  fixed apartment (hence we need only diagonal representatives), the formula of distance is quite simple:
\begin{prop}\label{propDistanceLoc}
Let $A,B\in \mathrm{Diag}^\times_d(\mathbb{Q}_p)$, define $\alpha_i=\nu(A_{i,i})$ and $\beta_i=\nu(B_{i,i})$, for $i=1,\ldots,d$. Then
\begin{align}\label{locHalfOfDis}
\begin{aligned}
m_{B:0}^{A:d}-m_{B:1}
^{A:d-1}&=\max_{i}\{\alpha_i-\beta_i\},\\ 
m_{B:d}^{A:0}-m_{B:d-1}
^{A:1}&=\max_{i}\{\beta_i-\alpha_i\}.
\end{aligned}
\end{align}
Consequently
\begin{align}\label{vyznamVzdalenosti}
\begin{aligned}
\delta(A_v,B_v)&=\max_{i=1}^{d}\{\alpha_i-\beta_i\}+\max_{i=1}^{d}\{\beta_i-\alpha_i\}.
\end{aligned}
\end{align}
\end{prop}
\begin{proof}
The formulas follows from this computation:
\begin{align*}
&m_{B:0}^{A:d}-m_{B:1}^{A:d-1}=
(\sum_j \alpha_j) - \min_{i=1}^d\{(\sum_{j\neq i} \alpha_j)+\beta_i\}=\\
=&-\min_{i=1}^d\{-(\sum_j \alpha_j)+
(\sum_{j\neq i} \alpha_j)+\beta_i\}
=-\min_{i=1}^d\{\beta_i-\alpha_{i}\}=\max_{i=1}^d\{\alpha_i-\beta_i\}.
\end{align*}
Where the first equality follows directly from definition of $m_{B:i}^{A:d-i}$'s. 
\end{proof}

\section{Distance formulas for more than two vertices}\label{sec7}
Let $\alpha=A_v$, $\beta=B_v$ and $\gamma=C_v$ be any three vertices in the building $\mathscr{B}$. Consider the following formula
$$\delta(\alpha,\beta)+\delta(\beta,\gamma)+\delta(\gamma,\alpha)=$$
$$ =\mathbf{m^{A:0}_{B:d}}+m^{A:d}_{B:0}\mathbf{-m^{A:1}_{B:d-1}}-m^{A:d-1}_{B:1}+\mathbf{m^{B:0}_{C:d}}+m^{B:d}_{C:0}\mathbf{-m^{B:1}_{C:d-1}}-m^{B:d-1}_{C:1}+$$
$$+\mathbf{m^{C:0}_{A:d}}+m^{C:d}_{A:0}\mathbf{-m^{C:1}_{A:d-1}}-m^{C:d-1}_{A:1}.$$
If we take only the odd summands (emphasized by the bold font), we get a ``mixed" formula:
$$ m^{A:0}_{B:d}+m^{B:0}_{C:d}+m^{C:0}_{A:d}-m^{A:1}_{B:d-1}-m^{B:1}_{C:d-1}-m^{C:1}_{A:d-1},$$
which satisfies the invariance conditions in all three components (i.e., well defines function, which may has geometric meaning). This example had motivated this section, where we shall investigate distance formulas for more than two vertices.

The following lemma is for the section crucial. Later we will see (through its applications), that it has triangle inequality-like meaning.
\begin{lemma}\label{ineq}
Let $\alpha=A_v$, $\beta=B_v$ and $\gamma=C_v$ be any three vertices. 
Following inequality is representatives-free and holds:
\begin{equation}\label{triangle}
m_{A:0}^{B:d}+m_{A:i}^{C:d-i}-m_{A:i}^{B:d-i}-m_{B:i}^{C:d-i}\geq 0.
\end{equation}
\end{lemma}
\begin{proof}
One can directly ensure, that the expression in equation~\eqref{triangle} satisfies the invariance conditions so it is representatives-free. Hence it is enough to proof the inequality for any special choice of the basis and matrix representatives. Chose the basis so that $\beta=I_v$. Further, by Lemma~\ref{localizedRepresentative} we can assume $A$ and $C$ are in localized form.

Let $\pi_A$ ($\pi_C$, resp.) be the permutation from Definition~\ref{localForm} for $A$ ($C$, resp.) and $\alpha_i=\nu(A_{i,\pi_A(i)})$ ($\gamma_i=\nu(C_{i,\pi_C(i)})$, resp.). Let $X$ be $\text{\footnotesize{(A:i,C:d-i)}}$-\textit{witness} (of $m_{A:i}^{C:d-i}$), composed from rows $\{A_{\iota,*}|\iota \in \mathcal{I}\}$ where $|\mathcal{I}|=i$ and $\{C_{\iota,*}|\iota \in \mathcal{J}\}$ where $|\mathcal{J}|=d-i$. Then:
\begin{equation}\label{eq:det}
\det(X)=\sum_{\pi\in S_d}sgn(\pi)X_{1, \pi(1)}
\cdot\cdots\cdot
 X_{d,\pi(d)}.
\end{equation}
Since the valuation of every element in row $A_{\iota,*}$ ($C_{\iota,*}$, resp.) is at least $\alpha_{\iota}$ ($\gamma_{\iota}$, resp.), there is a lower bound for summands in~(\ref{eq:det}):
\begin{equation}\label{eq05:minors}
\begin{aligned}
\nu(X_{1,\pi(1)}\cdot\cdots\cdot X_{d,\pi(d)})&=\sum_{k=1}^d\nu(X_{k,\pi(k)})\geq\\
\geq\sum_{\iota\in \mathcal{I}}\alpha_{\iota}+\sum_{\iota\in \mathcal{J}}\gamma_{\iota}&\geq \sum_{\iota=1}^{i}\alpha_{\iota}+\sum_{\iota= 1}^{d-i}\gamma_{\iota}=m_{A:i}^{B:d-i}+m_{B:i}^{C:d-i}.
\end{aligned}
\end{equation}
The second inequality in~\eqref{eq05:minors} holds because the $\alpha_i$'s respective $\gamma_i$'s are ordered by size and the last equality follows from Lemma~\ref{eq2.2} (recall $B=I$). 

Now, using triangle inequality for valuation and~\eqref{eq05:minors} we obtain 
$$m_{A:i}^{C:d-i}=|X|_p\geq m_{A:i}^{B:d-i}+m_{B:i}^{C:d-i}.$$ 
It remains to add $m_{A:0}^{B:d}=|I|_p=0$ to the left hand side and we are done.
\end{proof}
Previous lemma could be used to prove triangular inequality for formulas of distances (equations~\eqref{eqDistance} and~\eqref{elDis}) or for any $\delta^i(\alpha,\beta)= m_{0}^{d}+m_{d}^{0}-m_{i}^{d-i}-m_{d-i}^{i}$ where $i=0,\ldots,d$. This is just special case ($n=2$, $\pi=(1,2)$) of the following lemma:

\begin{prop}\label{generalLowerBound}
For each pair of vertices $A_v,B_v$ and integer $i=0,\ldots,d$ define $\delta^i(A_v,B_v):= m_{0}^{d}+m_{d}^{0}-m_{i}^{d-i}-m_{d-i}^{i}$. Then for each $n$-tuple of vertices $\bar{\alpha}=(\alpha_1,\ldots,\alpha_n)$, where $\alpha_s=(A_s)_v$, and any permutation $\pi\in S_n$ and a vertex $\beta=B_v$ we have
\begin{equation}\label{eqGeneralLowerBound}
\sum_s\delta^i(\alpha_s,\beta)\geq \sum_{s}\bigg(m^{A_{\pi(s)}:d}_{A_s:0}-m^{A_{\pi(s)}:i}_{A_s:d-i}\bigg).
\end{equation}
\end{prop}
\begin{proof}
\begin{align*}
\sum_{s}\delta^i(\alpha_s,\beta)&=\sum_{s}\bigg(m^{B:d}_{A_s:0}+m^{B:0}_{A_s:d}-m^{B:d-i}_{A_s:i}-m^{B:i}_{A_s:d-i}\bigg)=\\
&=\sum_{s}m^{B:0}_{A_s:d}+\sum_{s}\bigg(m^{B:d}_{A_s:0}-m^{B:d-i}_{A_{\pi(s)}:i}-m^{B:i}_{A_{s}:d-i}\bigg)\geq\\
&\geq\sum_{s}m^{B:0}_{A_s:d}+\sum_{s}-m_{A_{s}:d-i}^{A_{\pi(s)}:i}
=\sum_{s}\bigg(m^{A_{\pi(s)}:0}_{A_s:d}-m_{A_{s}:d-1}^{A_{\pi(s)}:i}\bigg).
\end{align*}
The inequality follows by applying inequality~\eqref{ineq} to all summands in the second summation. At the last step we used $m^{B:0}_{A_s:d}=|A_s|_p=m^{A_{\pi(s)}:0}_{A_s:d}$.

Notice that during the calculation we preserved the invariance conditions of first formula, hence the last formula satisfies invariance conditions as well.
\end{proof}

An essential knowledge in the proof of Theorem~\ref{generalDistance} is the well-known Hall's condition for existence of perfect matching (Theorem 2.1.2 in~\cite{Ha}).
\begin{thm}[Hall's condition]
Let $X=\{x_1,\ldots,x_m\}$ be any set and $\triangleright$ be binary relation on $X$. Then there is a permutation $\pi\in S_m$ such that $x_i\triangleright x_{\pi(i)}$ for each $i=1,\ldots,d$, if and only if for each $Y\subseteq X$ 
$$|\{x\in X|x\triangleright y, \text{\ for\ some\ }y\in Y\}|\geq |Y|.$$
\end{thm}

\begin{thm}\label{generalDistance}
Let all vertices of $n$-tuple $\bar{\alpha}=(\alpha^1,\ldots,\alpha^n)$ have common apartment $\mathscr{A}$ and $\beta$ be some vertex laying in $\mathscr{A}$ minimalizing
\begin{equation}
\delta_{\bar{\alpha}}(-):=\sum_i\delta(-,\alpha_i)
\end{equation}
among all the vertices of the apartment $\mathscr{A}$. Then there is a permutation $\pi\in S_n$ such that in inequality~\eqref{eqGeneralLowerBound} (for $i=1$) equality holds. Moreover any local minimum of $\delta_{\bar{\alpha}}$ in $\mathscr{A}$ is also a global minimum (meaning on $\mathscr{A}$).
\end{thm}
\begin{proof}
We can assume $\beta=I_v$ and every $\alpha^r$ has diagonal representative $A^r$ (the index is placed up in order to avoid double lower index), having entries in $\mathbb{Z}_p$ and at least one of zero valuation. Hence $A^r=\mathrm{diag}(A^r_{1,1},\ldots,A^r_{d,d})$ and for $\alpha^r_i:=\nu(A^r_{i,i})$ we have $\min_{1\leq i\leq d}\{\alpha_i^r\}=0$. For every $r=1,\ldots,n$ define 
$I_{max}^{r},I_{min}^{r}\subseteq \{1,\ldots,d\}$ such that: 
$$j\in I_{max}^{r}\Leftrightarrow \alpha_{j}^{r} = \max_{i}\{\alpha_{i}^{r}\}\text{\ \ and\ \ }j\in I_{min}^{r}\Leftrightarrow \alpha_{j}^{r} = 0=\min_{i}\{\alpha_{i}^{r}\}.$$
Notice that by the formulas from Proposition~\ref{propDistanceLoc} we have
\begin{equation}
\delta(\alpha^r,\beta)=\delta(\alpha^r,I_v)=\alpha^r_j, \text{\ for\ each\ }j\in I^r_{max}.
\end{equation}

Next define a binary relation $\vartriangleright$ on $A^r$'s by prescription
\begin{equation}\label{definitionOfRelation}
A^r \vartriangleright A^s \overset{def}{\Leftrightarrow} m_{A^s:d}^{A^r:0}-m_{A^s:d-1}^{A^r:1}=\delta(\beta,\alpha^s)\Leftrightarrow  I^{s}_{max}\bigcap I^{r}_{min}\neq \emptyset.
\end{equation} 
The second equivalence follows by applying Proposition~\ref{propDistanceLoc}:
$$\delta(\beta,\alpha^s)=\max\{\alpha_j^s\} \text{\ \ and\ \ }
m_{A^s:d}^{A^r:0}-m_{A^s:d-1}^{A^r:1}=\max_j\{\alpha_{j}^s-\alpha_{j}^{r}\}.$$
Hence the equality holds if and only if there is $j$ such that $\alpha_{j}^{s}=\max_{j}\{\alpha_{j}^{s}\}$ 
and $\alpha_{j}^{r}=0$, equivalently $j\in I^{s}_{max}\bigcap I^{r}_{min}\neq \emptyset$. 

Let us parametrize all vertices in neighbourhood of $\beta=I_v$ in the apartment by $\{\beta^{I}|I\subseteq\{1,\ldots,d\}\}$, where $\beta^I$ has diagonal matrix representative $B^I$ having only $1$'s and $p$'s on diagonal, as follows:
$$B^I=\mathrm{diag}(p^{\beta^{I}_{1}},\ldots,p^{\beta^{I}_{d}}) \text{\ \ where\ \ }
\beta^I_j=\begin{cases} 
      \hfill 1 \hfill & \text{\ if\ } j\in I,\\
      \hfill 0 \hfill & \text{\ if\ } j\not\in I.\\
  \end{cases}
$$
Of course $\beta^{\emptyset}=\beta^{\{1,\ldots,d\}}=\beta$.

Let us see what is the value of $\delta_{\bar{\alpha}}$ on $\beta^I$'s. Using formula~\eqref{vyznamVzdalenosti} we get
$$\delta_{\bar{\alpha}}(\beta^I)=\sum_{r}\max_{j}\{\alpha_{j}^{r}-\beta^{I}_{j}\}+\sum_{r}\max_{j}\{\beta^{I}_{j}-\alpha_{j}^{r}\}.$$
When we change $I=\emptyset$ to $I\neq \emptyset$, every $\max_{j}\{\alpha_{j}^{r}-\beta^{I}_{j}\}$ stays the same or decrease by $-1$ if $I^r_{max} \subseteq I$ and every $\max_{j}\{\beta^{I}_{j}-\alpha_{j}^{r}\}$ stays the same or increase by $+1$ if $I^r_{min}\cap I \neq \emptyset$. Because of the minimality of $\delta_{\bar{\alpha}}(\beta)$, for every $I$ we have
\begin{equation}\label{eqThisMayLastEq}
|\{r|I^r_{max}\subseteq I\}| \leq |\{r|I^r_{min}\cap I\not =\emptyset\}|.
\end{equation}
Now choose any $S\subseteq\{1,\ldots,n\}$ and define $I=\bigcup_{s\in S}I^{s}_{max}$. According to~\eqref{eqThisMayLastEq} there has to be at least $|S|$ indices $r$'s such that $I_{min}^{r}$ satisfies $I_{min}^{r}\cap I\neq \emptyset$. According to~\eqref{definitionOfRelation} that means there are at least $|S|$ matrices $A^r$'s being in relation $A^r \vartriangleright \{A^s|s\in S\}$. That is exactly the Hall's condition for the existence of permutation $\pi$ satisfying $A^r \vartriangleright A^{\pi(r)}$ for all $r$'s. From~\eqref{definitionOfRelation} using $\pi$ we achieve the equality in~\eqref{eqGeneralLowerBound}.

Since we have used only local minimality of $\delta_{\bar{\alpha}}(-)$ at $\beta$ and we already get on the button of the lower bound~\eqref{eqGeneralLowerBound}, at $\beta$ the global minimum occurs as well. 
\end{proof}
Putting $i=1$, a geometric interpretation of the LH'S of~\eqref{eqGeneralLowerBound} is clear. We can also provide a geometric interpretation of the RH'S.
Recall Example~\ref{exDirectedGraph} where we were investigated the distance formula in the restricted directed graph $\mathscr{B}[1]$. In $\mathscr{B}[1]$ let us consider any collection of cycles, which union pass through all vertices $\alpha_1,\ldots,\alpha_n$ and hence naturally defines permutation $\pi\in S_n$. If it connects each pair of consecutive vertices $\alpha_i,\alpha_{\pi(i)}$ by a minimal path, the length of the cycle equals according to~\eqref{directedDistance} (where actually equality holds)
$$ \sum_s\overrightarrow{\delta}(\alpha_s,\alpha_{\pi(s)})=d\cdot\sum_s \bigg(|A_s|-m_{A_s:1}^{A_{\pi(s)}:d-1}\bigg).$$
Because of this geometric interpretation we realize, that we can always find a permutation $\pi$ maximalizing RH'S of~\eqref{eqGeneralLowerBound} for $i=1$ without fixed points. Since by incorporating a fixed point of $\pi$ into any cycle of $\pi$, the RH'S for the resulting permutation can't decrease by the triangular inequality for $\overrightarrow{\delta}$. 
Specially, in the case $n=3$, only cycles matter:
\begin{cor}
The length of a minimal tree connecting three vertices $A_v$, $B_v$ and $C_v$ is greater or equal to $\max\{\Lambda_1,\Lambda_2\}$ where
$$ \Lambda_1=m^{A:0}_{B:d}+m^{B:0}_{C:d}+m^{C:0}_{A:d}-m^{A:1}_{B:d-1}-m^{B:1}_{C:d-1}-m^{C:1}_{A:d-1},$$
$$\Lambda_2= m^{A:0}_{C:d}+m^{C:0}_{B:d}+m^{B:0}_{C:d}-m^{A:1}_{C:d-1}-m^{C:1}_{B:d-1}-m^{B:1}_{A:d-1}.$$
\end{cor}
\begin{rem}
Denote by $\Lambda$ the length of a minimal tree connecting some vertices $A_v$, $B_v$ and $C_v$. A trivial lower bound for $\Lambda$ is $\frac{1}{2}[\delta(A_v,B_v)+\delta(B_v,C_v)+\delta(C_v,A_v,)]$, that follows (for any graph) directly from triangular inequality. See that $\delta(A_v,B_v)+\delta(B_v,C_v)+\delta(C_v,A_v,)=\Lambda_1+\Lambda_2$ (recall the introduction to this section). Hence the lower bound $\Lambda\geq\max\{\Lambda_1,\Lambda_2\}$ is stronger than the trivial one $\Lambda\geq\frac{\Lambda_1+\Lambda_2}{2}$.
\end{rem}
\section{Appendix}
Let us try compute intersection of two apartments in the case $d=2$, using the formula of distance (Theorem~\ref{distance}). Given an apartment $$\mathscr{A}=\left( \begin{array}{ccc}   
    c & d\\
    e & f\\
    \end{array} \right)_a,$$ 
where $c,d,e,f\in \mathbb{Q}_p$, what is the intersection of $\mathscr{A}$ and $I_a$? Consider following pair of vertices 
$$\alpha_i=\left( \begin{array}{ccc}   
    p^i & 0\\
    0 & 1\\
    \end{array} \right)_v, \beta_j=\left( \begin{array}{ccc}   
    p^jc & p^jd\\
    e & f\\
    \end{array} \right)_v,$$
for $i,j\in\mathbb{Z}$. We are going to solve equation 
\begin{equation}\label{solvedEqu}
\delta(\alpha_i,\beta_j)=0,
\end{equation} with respect to unknown integers $i,j$. We are interested rather in the inequality $\delta(\alpha_i,\beta_j)\leq 0$, since the other direction holds for each $i,j$. Now apply Theorem~\ref{distance}:
$$2\cdot \min\bigg\{
\left|\begin{array}{ccc}   
    p^jc & p^jd\\
    p^i & 0\\
    \end{array} \right|_p,\left| \begin{array}{ccc}   
    p^jc & p^jd\\
    0 & 1\\
    \end{array} \right|_p,\left| \begin{array}{ccc}   
    e & f\\    
    p^i & 0\\
    \end{array} \right|_p,\left| \begin{array}{ccc}   
    e & f\\
    0 & 1\\
    \end{array} \right|_p
\bigg\} \geq$$
$$\geq\left|\begin{array}{ccc}   
    p^jc & p^jd\\
    e & f\\
    \end{array} \right|_p+
    \left| \begin{array}{ccc}   
    p^i & 0\\
    0 & 1\\
    \end{array} \right|_p.$$
After evaluation of the determinants: 
$$ 2\cdot\min\{i+j+\nu(d),j+\nu(c),i+\nu(f),\nu(e)\}\geq i+j+\nu(cf-de).$$
Which splits into four inequalities:
\begin{align}
i+j +2\cdot\nu(d)\geq\nu(cf-de),\label{1} \\
 j-i +2\cdot\nu(c)\geq \nu(cf-de), \label{2}\\
i-j +2\cdot\nu(f)\geq \nu(cf-de), \label{3}\\
-i-j +2\cdot\nu(e)\geq \nu(cf-de). \label{4}
\end{align}

Summing \eqref{1}+\eqref{4} we get $\nu(de)=\nu(d)+\nu(e)\geq \nu(cf-de)$ and similarly 
\eqref{2}+\eqref{3} leads to $\nu(cf)=\nu(c)+\nu(f)\geq \nu(cf-de)$. Combine these inequalities with triangular inequality and realize
\begin{equation}
\min\{\nu(de),\nu(cf)\}\geq \nu(cf-de)\geq \min\{\nu(de),\nu(cf)\}.
\end{equation}
Hence either holds: there is no pair of integers $i,j$ solving the equality~\eqref{solvedEqu} or 
\begin{equation}\label{asA}
\min\{\nu(cf),\nu(de)\}=\nu(cf-de).
\end{equation}
Suppose the last case holds and $\nu(cf-de)=\nu(cf)$. Substitute this to inequalities~\eqref{2} and~\eqref{3}, after some corrections we get
\begin{align}
 j-i \geq \nu(f)-\nu(c),\\
i-j \geq \nu(c)-\nu(f).
\end{align}
From which we obtain $j-i=\nu(f)-\nu(c)$. Then substitute $j=i+\nu(f)-\nu(c)$ into~\eqref{1} and~\eqref{4}, we get 
\begin{align}
i+i-\nu(c)+\nu(f)+2\cdot\nu(d)\geq\nu(c)+\nu(f),\\
-i-i+\nu(c)-\nu(f) +2\cdot\nu(e)\geq \nu(c)+\nu(f).
\end{align}
These two inequalities lead to the solution of~\eqref{solvedEqu}:
\begin{equation}\label{sol1}
\nu(c)-\nu(d)\leq i\leq \nu(e)-\nu(f), j=i-\nu(c)+\nu(f).
\end{equation} 
Geometrically this means, that the intersection of apartments is empty or one vertex or a path of the length $\nu(e)-\nu(f)-\nu(c)+\nu(d)=\nu(de)-\nu(cf)$ which is non-negative as we have supposed $\min\{\nu(cf),\nu(de)\}=\nu(cf)$. 

If the other case holds: $\nu(cf-de)=\nu(d)+\nu(e)$, we can change the representative of $\mathscr
{A}$ by interject its rows and overwrite the previous solution with $c\leftrightarrow e$ and $d\leftrightarrow f$:
\begin{equation}\label{sol2}
\nu(e)-\nu(f)\leq i\leq \nu(c)-\nu(d), j=i-\nu(e)+\nu(d).
\end{equation}
Now we can put the discussed cases together into one formula for number of edges shared by $\mathscr{A}$ and $I_a$. If $\nu(cf-de)= \min\{\nu(cd),\nu(de)\}$, according to~\eqref{sol1} and~\eqref{sol2} there is non-empty intersection, with
\begin{equation}\label{forInt}
|\nu(e)-\nu(c)+\nu(d)-\nu(f)|
\end{equation}
edges. If $\nu(cf-de)\gneq \min\{\nu(cd),\nu(de)\}$, the apartments have empty intersection. However, this could happen only if $\nu(cf)=\nu(de)$, which means the formula~\eqref{forInt} vanishes.

\begin{thm}
For any pair of apartments in $\mathscr{B}(\mathrm{SL}_2(\mathbb{Q}_p))$
$$\mathscr{A}'=\left( \begin{array}{ccc}   
    c & d\\
    e & f\\
    \end{array} \right)_a,\ 
\mathscr{A}''=\left( \begin{array}{ccc}   
    r & s\\
    t & u\\
    \end{array} \right)_a  
$$
the number of edges they share is equal to
\begin{equation}\label{formulaOfSharedEdges}
\bigg|     
     \left|\begin{array}{ccc}   
    c & d\\
    r & s\\
    \end{array} \right|_p-
     \left|\begin{array}{ccc}      
    c & d\\
    t & u\\
    \end{array} \right|_p+
     \left|\begin{array}{ccc}   
    e & f\\
    t & u\\
    \end{array} \right|_p-
     \left|\begin{array}{ccc}   
    e & f\\
    r & s\\
    \end{array} \right|_p    
    \bigg|.
\end{equation}
\end{thm}
\begin{proof}
If we act on $\mathscr{A}'$ and $\mathscr{A}''$ by $T\in \mathrm{GL}_2(\mathbb{Q}_p)$ (that is the matrix representatives are multiplied by $T$), the formula stays the same:
$$
\bigg|     
     \left|\begin{array}{ccc}   
    c & d\\
    r & s\\
    \end{array} \right|_p+|T|_p-
     \left|\begin{array}{ccc}      
    c & d\\
    t & u\\
    \end{array} \right|_p-|T|_p+
     \left|\begin{array}{ccc}   
    e & f\\
    t & u\\
    \end{array} \right|_p+|T|_p-
     \left|\begin{array}{ccc}   
    e & f\\
    r & s\\
    \end{array} \right|_p-|T|_p    
    \bigg|.
$$
Hence we can by an action of a convenient $T$ transform $\mathscr{A}''$ to $I_a$, and so we can assume $\mathscr{A}''=I_a$.

Evaluating formula~\eqref{formulaOfSharedEdges} in this case we get
$$
\bigg|     
     \left|\begin{array}{ccc}   
    c & d\\
    1 & 0\\
    \end{array} \right|_p-
     \left|\begin{array}{ccc}      
    c & d\\
    0 & 1\\
    \end{array} \right|_p+
     \left|\begin{array}{ccc}   
    e & f\\
    0 & 1\\
    \end{array} \right|_p-
     \left|\begin{array}{ccc}   
    e & f\\
    1 & 0\\
    \end{array} \right|_p    
    \bigg|.
$$
That equals $|\nu(d)-\nu(c)+\nu(e)-\nu(f)|$, and so the statement holds according to the foregoing computation.
\end{proof}
\begin{rem}
Note that neither multiplying any row of the two matrices by scalar nor interjecting rows of any matrix changes the value of formula~\eqref{formulaOfSharedEdges}.
\end{rem}
\section{Conclusions}
We have described how to code the vertices of $\mathscr{B}(\mathrm{SL}_d(\mathbb{Q}_p))$ by certain matrix representatives. Then we have introduced so called distance formulas: we showed that using valuations of certain determinants, $+$, $-$ and $\mathrm{max}$ we can compose well-defined functions on pairs of vertices. An explicit formula of graph distance is a key example of the concept (some others are mentioned). At the first sight the distance formulas are not computation-friendly, however knowing the form or even the existence of an explicit formula is interesting. We have also studied generalization of distance formulas to three and more vertices. We have partially succeeded in generalizing formula of graph distance to certain inequality, (where equality holds when we restrict ourself to the cease of one apartment). We expects the equality holds in the general case as well.   
In Appendix we performed calculation which led to formula counting the number of edges common to two given apartments in $\mathscr{B}(\mathrm{SL}_2(\mathbb{Q}_p))$. We have used the distance formula (Theorem~\ref{distance}) to transform the problem to inequalities in comfortable way. Nevertheless, we could arrive at the inequalities without use of the distance formula.

\section{Acknowledgment}
Most of the results in the article are part of my master thesis. Hence, I would like to thanks Vít\v{e}zslav Kala, my supervisor during writing the thesis. I appreciates his valuable advices as well as his excellent work with students.

\bibliography{bibliography}

\bibliographystyle{Alpha}
\end{document}